\documentclass[a4paper,12pt,reqno]{amsart}

\usepackage{amsmath}
\usepackage{amscd}
\usepackage{amssymb}
\usepackage{mathrsfs}
\usepackage[left=2.5cm,right=2.5cm,bottom=3cm,top=3cm]{geometry}
\newtheorem{thm}{Theorem}[section]
\newtheorem{cor}[thm]{Corollary}

\newtheorem{lem}[thm]{Lemma}

\newtheorem{ques}[thm]{Question}
\newtheorem{prop}[thm]{Proposition}

\theoremstyle{definition}
\newtheorem{defn}[thm]{Definition}
\newtheorem{rem}[thm]{Remark}

\numberwithin{equation}{section}

\begin{document}

\title[Local Curvature Estimates of Long-Time Solutions to KRF]{Local Curvature Estimates of Long-Time Solutions to the K\"ahler-Ricci Flow}

\author{Frederick Tsz-Ho Fong}
\address{Department of Mathematics, Hong Kong University of Science and Technology, Clear Water Bay, Sai Kung, Kowloon, Hong Kong S.A.R., China}
\email{frederick.fong@ust.hk}
\thanks{Fong is partially supported by Hong Kong Research Grant Council Early Career Scheme \#26301316 and General Research Fund \#16300018}

\author{Yashan Zhang}
\address{School of Mathematics and Hunan Province Key Lab of Intelligent Information Processing and 
Applied Mathematics, Hunan University, Changsha 410082, China}
\email{yashanzh@hnu.edu.cn}
\thanks{Zhang is partially supported by Fundamental Research Funds for the Central Universities (No. 531118010468) and National Natural Science Foundation of China (No. 12001179)}

\begin{abstract}
We study the local curvature estimates of long-time solutions to the normalized K\"ahler-Ricci flow on compact K\"ahler manifolds with semi-ample canonical line bundle. Using these estimates, we prove that on such a manifold, the set of singular fibers of the semi-ample fibration on which the Riemann curvature blows up at time-infinity is independent of the choice of the initial K\"ahler metric. Moreover, when a regular fiber of the semi-ample fibration is not a finite quotient of a torus, we determine the exact curvature blow-up rate of the K\"ahler-Ricci flow near the regular fiber.
\end{abstract}

\renewcommand{\subjclassname}{%
  \textup{2010} Mathematics Subject Classification}

\maketitle

\section{Introduction}

Let $X$ be an $n$-dimensional compact K\"ahler manifold. We study the solution  $\omega=\omega(t)$, $t\in[0,\infty)$, to the (normalized) K\"ahler-Ricci flow
\begin{equation}\label{nkrf}
\partial_t\omega=-Ric(\omega)-\omega
\end{equation}
starting from any K\"ahler metric $\omega_0$ on $X$.

The maximal existence time theorem of the K\"ahler-Ricci flow by Cao, Tsuji, Tian-Z.Zhang \cite{C,TZz,Ts} showed that the existence of long time solutions of \eqref{nkrf} is equivalent to nefness of the canonical line bundle $K_X$. The Abundance Conjecture predicts that if the canonical line bundle of an algebraic manifold is nef (numerically effective), then it is semi-ample. Hence, it is natural to study the K\"ahler-Ricci flow  on $n$-dimensional compact K\"ahler manifolds with semi-ample canonical line bundles. Let $X$ be such a manifold, the convergence and singular behaviors of the K\"ahler-Ricci flow were extensively studied by various authors including Tian-Z.Zhang \cite{TZz}, Song-Tian \cite{ST07,ST12,ST16}, Z.Zhang \cite{Z}, Gill \cite{Gi}, the first-named author and Z.Zhang \cite{FZ}, Tosatti-Weinkove-Yang \cite{TWY}, Tosatti-Y.G.Zhang \cite{ToZyg}, Hein-Tosatti \cite{HeTo}, Guo-Song-Weinkove \cite{GSW}, Tian-Z.L.Zhang \cite{TZzl16,TZzl18}, Guo \cite{G}, the second-named author \cite{Zys17, Zys18, Zys19} and Jian \cite{Jian}.

When the Kodaira dimension $\kappa := kod(X)$ of $X$ is in the range $0<\kappa<n$, we let
\begin{equation}\label{semiample}
f:X\to X_{can}\subset \mathbb{CP}^N
\end{equation}
be the semi-ample fibration with connected fibers induced by pluricanonical system of $K_X$. Here $X_{can}$ is a $\kappa$-dimensional irreducible normal projective variety which is called the canonical model of $X$. Let $V\subset X_{can}$ be the singular set of $X_{can}$ consisting of critical values of $f$. We call $X_y=f^{-1}(y)$ is a regular fiber if $y\in X_{can}\setminus V$, and a singular fiber if $y \in V$. Moreover, there exists a rational K\"ahler metric $\chi$ on $\mathbb{CP}^N$ with $f^*\chi\in2\pi c_1(K_X)$. By fundamental works of Song-Tian \cite{ST07, ST12}, the K\"ahler-Ricci flow $\omega(t)$ starting from any initial K\"ahler metric converges in the sense of currents to a generalized K\"ahler-Einstein metric on $X_{can}$. Smooth convergence results were obtained by \cite{Gi} in the case $X$ is a direct product of a complex torus and a K\"ahler manifold with negative first Chern class (also see \cite[Section 6]{SW} for the product elliptic surface case), and by \cite{FZ} for the case of regular torus fibrations (proved using parabolic analogue of \cite{GTZ}). The rationality assumption in \cite{FZ} was later removed in \cite{HeTo}. For general Calabi-Yau fibrations, $C_{loc}^0$-convergence of the metric on the regular part of $X$ was also obtained in \cite{TWY}. It was also proved in \cite{ToZyg} that the Riemann curvature must blow up near any regular fiber which is not a complex torus nor its finite quotient.

In this article, motivated by known results about curvature estimates of long-time K\"ahler-Ricci flow solutions, we study the \emph{local} curvature estimate of the flow \eqref{nkrf}. We first localize Hamilton's definition \cite{Ha93} of infinite-time singularity types of the K\"ahler-Ricci flow:

\begin{defn}[Local infinite-time singularity type]\label{defn_1}
Given a subset $K\subset X$, a long time solution $\omega(t)$ to the K\"ahler-Ricci flow \eqref{nkrf} is of \emph{singularity type \textrm{III} on $K$} if there exists an open neighborhood $U$ of $K$ such that
\begin{equation}\label{defn_ii}
\limsup_{t\to\infty}\left(\sup_{U}|Rm(\omega(t))|_{\omega(t)}\right)<\infty\nonumber;
\end{equation}
otherwise we say the solution is of \emph{singularity type \textrm{IIb} on $K$}.
\end{defn}

For example, if we choose $K=X$ in Definition \ref{defn_1}, then we get the original definition by Hamilton \cite{Ha93} and we simply say the solution $\omega(t)$ is of type III or type IIb. We can also discuss the singularity type of the K\"ahler-Ricci flow at a fixed point $x$ by choosing $K=\{x\}$ in Definition \ref{defn_1}.

\par We are interested in classifying \emph{the singularity type of the K\"ahler-Ricci flow on the fibers $X_y$ of $f$ in \eqref{semiample}}. Thanks to the aforesaid works \cite{FZ,Gi,HeTo,TWY,ToZyg}, the regular fiber case was completely understood: the K\"ahler-Ricci flow is of type \textrm{III} on a regular fiber $X_y$ if and only if $X_y$ is biholomorphic to a finite quotient of a torus. Moreover, in the case that the regular fiber is not a finite quotient of a torus, any open neighborhood of a singular fiber must contain a regular fiber, and so in this case the K\"ahler-Ricci flow on a singular fiber $X_y$ is of type \textrm{IIb}. Therefore, the only open case is the following

\vskip 0.3cm
\begin{itemize}
\item[($\star$)]\emph{$0<kod(X)<n$ ($n\ge3$) and the regular fiber $X_y$ is a finite quotient of a torus and $V\neq\emptyset$}.
\end{itemize}
\vskip 0.3cm

In case ($\star$), when $X_y$ is a singular fiber, certain criterions for the K\"ahler-Ricci flow developing type \textrm{IIb} singularities on $X_y$ have been discovered, which relate the singularity type of the K\"ahler-Ricci flow to certain algebraic properties of $X_y$, see \cite[Proposition 1.4]{ToZyg} and \cite[Theorems 2.1]{Zys19}. In general, the classification of the singularity type of the K\"ahler-Ricci flow on singular fibers is largely open. Along this line, it was conjectured in \cite[Section 1]{ToZyg} (also see \cite[Conjecture 6.7]{To.kawa}) and confirmed in \cite[Theorem 1.4]{Zys17} that the (global) singularity type of the K\"ahler-Ricci flow on $X$ does not depend on the choice of the initial K\"ahler metric. This indicates that the (global) singularity type of the K\"ahler-Ricci flow on $X$ should only depend on the complex structure of $X$.

\par Now suppose $X$ in case ($\star$) admits a type \textrm{II}b solution to the K\"ahler-Ricci flow on $X$, then \cite[Theorem 1.4]{Zys17} implies every solution to the K\"ahler-Ricci flow on $X$ is of type \textrm{II}b, that is, the curvature of every solution to \eqref{nkrf} must blow up on \emph{some} singular fibers. It is then natural to ask:
\begin{ques}
\label{mainq}
Assume the flow is of type \textrm{II}b on $X$.
\begin{itemize}
\item[(1)] For $X$ in case $(\star)$ whether the set of singular fibers on which the Riemann curvature of the K\"ahler-Ricci flow \eqref{nkrf} blows up is independent of the choice of the initial K\"ahler metric? This in its setting may be regarded as a local and strengthened version of a conjecture by Tosatti \cite[Conjecture 6.7]{To.kawa}.
\item[(2)] Furthermore, if the flow is of type \textrm{II}b on a singular or regular fiber $X_y$, what can we say about the blow-up rate of the curvature? E.g.,
\begin{itemize}
\item[(2.1)] does the blow-up rate depend on the choice of the initial K\"ahler metric?
\item[(2.2)] can we have some effective estimates on the blow-up rates?
\end{itemize}
\end{itemize}
\end{ques}

In this paper, we shall study the above questions. We will first prove the following general estimates assuming that both curvature and metrics satisfy certain time-dependent bounds:
\begin{thm}\label{tech.thm}
Let $f:X\to X_{can}$ be the map in \eqref{semiample}, consider a fiber $X_y$ where $y \in X_{can}$. Suppose $\tilde\omega(t)$ and $\omega(t)$ are two solutions of \eqref{nkrf} such that there exist an open neighborhood $U$ of $X_y$ in $X$, and two increasing\footnote{Throughout this article, ``increasing'' means the time-derivative is nonnegative.} differentiable functions $\tau(t), \sigma(t) : [0, \infty) \to [1, \infty)$ such that
\begin{equation}\label{tau}
\sup_{U}|Rm(\tilde\omega(t))|_{\tilde\omega(t)}\le\tau(t) \quad \text{ for any } t \in [0,\infty),
\end{equation}
and, we have on $U\times[0,\infty)$,
\begin{equation}\label{sigma}
\sigma(t)^{-1}\tilde\omega(t)\le\omega(t)\le\sigma(t)\tilde\omega(t).
\end{equation}
Then, there exist an open neighborhood $U'$ of $X_y$ in $X$ with $U'\subset\subset U$ and a constant $A\ge1$ such that for any $t\in[0,\infty)$,
\begin{equation}\label{type_iii.1.1.1}
\sup_{U'}|Rm(\omega(t))|_{\omega(t)}\le A\sigma(t)^4\tau(t).
\end{equation}
\end{thm}
We point out that in Theorem \ref{tech.thm} the fiber could be a \emph{singular} one.
\par Combining Theorem \ref{tech.thm}, early results in \cite{FZ} and Proposition \ref{equivalence} in Section \ref{sect:MetricEqv}, we have the following estimates assuming only the curvature upper bound: 
\begin{thm}\label{blowuprate}
Let $f:X\to X_{can}$ be the map in \eqref{semiample}, and $\omega(t)$ and $\tilde{\omega}(t)$ two K\"ahler-Ricci flows \eqref{nkrf}. Consider a fiber $X_y$, and suppose there exist an open set $U$ containing $X_y$ and an increasing differentiable function $\tau(t) : [0, \infty) \to [1, \infty)$ such that
\[\sup_U |Rm(\tilde{\omega}(t))|_{\tilde{\omega}(t)} \leq \tau(t).\]
\begin{itemize}
\item[(i)] If $X_y$ is a regular fiber, then there exist an open set $U'$ with $X_y \subset U' \subset\subset U$ and a constant $C\ge1$ such that
\[\sup_{U'} |Rm(\omega(t))|_{\omega(t)} \leq C\tau(t).\]
\item[(ii)] If $X_y$ is a singular fiber, then there exist an open set $U'$ with $X_y \subset U' \subset\subset U$ and a constant $C\ge1$ such that
\[\sup_{U'} |Rm(\omega(t))|_{\omega(t)} \leq e^{C\tau(t)}.\]
\end{itemize}
\end{thm}

\begin{rem}
Though our conclusions in Theorems \ref{tech.thm} and \ref{blowuprate} are local, we should mention that the involved constants $A$ and $C$ depend on some global properties of the initial metrics $\omega_0$ and $\tilde\omega_0$. Indeed, our arguments are heavily based on Song-Tian's uniform $C^0$-estimates on the K\"ahler potentials and the parabolic Schwarz lemma along the K\"ahler-Ricci flow (see \cite{ST07,ST12,ST16}), in which the uniform estimates depend on global properties of the initial metric $\omega_0$, including its $C^0$-norm, integrals $\int_X\omega_0^n$, $\int_X\omega^{n-k}\wedge f^*\chi^k$, and the curvature bound of $\chi$. Therefore, the constants $A$ and $C$ also depend on these informations of initial metrics and $\chi$.
\end{rem}
Both conclusions in Theorem \ref{blowuprate} can be seen as partial results for Question \ref{mainq}(2), while part (i) solves Question \ref{mainq} (2.1) in the regular fiber case. Moreover, as a corollary of Theorem \ref{blowuprate}, we shall answer Question \ref{mainq} (1) in full generality:

\begin{thm}\label{ind.1}
Let $f:X\to X_{can}$ be the map in \eqref{semiample}. The (local) singularity type (i.e. Type III or IIb) of the K\"ahler-Ricci flow \eqref{nkrf} on any fixed fiber $X_{y}$ does not depend on the choice of the initial K\"ahler metric.

Consequently, if let $W_{\textup{II}}(\omega(t))$ be the set of point $y\in X_{can}$ such that $\omega(t)$ is of type \textrm{II}b on $X_y$, then $W_{\textup{II}}(\omega(t))$ and hence $f^{-1}(W_{\textup{II}}(\omega(t)))$ are invariant for different K\"ahler-Ricci flow solutions $\omega(t)$ on $X$. We may simply denote $W_{\textup{II}} := W_{\textup{II}}(\omega(t))$.
\end{thm}

\begin{rem}
In \cite{FZ}, it was proved that $W_{\textup{II}} \subset V$ in the case $(\star)$. If in particular $X$ is a minimal elliptic K\"ahler surface, by \cite[Theorem 1.6]{ToZyg} the set $W_{\textup{II}}$ exactly consists of the critical values of $f$ over which the singular fibers are not of Kodaira type $mI_0$. In contrast in \cite{ToZyg} where the regular fibers are not complex tori nor its finite quotients, one has $(X_{can} \backslash V) \subset W_{\textup{II}}$; moreover, as we mentioned before, in this case any open neighborhood of a singular fiber must contain a regular fiber, and so we have $W_{\textup{II}}=X_{can}$. 
\end{rem}

\begin{rem}
Note that \cite[Theorem 1.4]{Zys17} addressed the special case $W_{\textrm{II}}(\omega(t)) = \emptyset$ for some solution $\omega(t)$. Therefore, our Theorem \ref{ind.1} is a substantial improvement of \cite[Theorem 1.4]{Zys17}.
\end{rem}

\begin{rem}
Theorem \ref{ind.1} indicates a general phenomenon that the singularity type of the K\"ahler-Ricci flow on a singular fiber should depend only on the properties of the singular fiber itself, and hence provides an analytic viewpoint to classify these singular fibers. Naturally, our long term goal will be to find precise relations between singularity types of K\"ahler-Ricci flows and the analytic/algebraic properties of the singular fibers.
\end{rem}

\par To state our next result, consider the regular fiber of $f:X\to X_{can}$ is not biholomorphic to a finite quotient of a torus, then it was proved in \cite{ToZyg} that the Riemann curvature must blow up near such a fiber. In this case our Theorem \ref{blowuprate}(i) somehow says the curvature blow-up rate near such a fiber should not depend on the choice of the initial metric. The following theorem further determines the exact blow-up order of the Riemann curvature, solving Question \ref{mainq} (2) in the regular fiber case.

\begin{thm}\label{thm_Rm}
Assume $f:X\to X_{can}$ as in \eqref{semiample} and the regular fiber is not biholomorphic to a finite quotient of a torus. Then for an arbitrary $U\subset\subset X\setminus f^{-1}(V)$ and an arbitrary solution $\omega(t)$ to the K\"ahler-Ricci flow \eqref{nkrf}, there is a constant $A\ge1$ such that for any $t\ge0$,
\begin{equation}\label{thm_Rm_e^t}
A^{-1}e^{t}\le\sup_{U}|Rm(\omega(t))|_{\omega(t)}\le Ae^{t}.
\end{equation}
\end{thm}

We should mention that in \eqref{thm_Rm_e^t} the lower bound is essentially contained in \cite{ToZyg}, and so our contribution is the upper bound, see Section \ref{sect:e^t} for more details.
\par An immediate consequence of Theorem \ref{thm_Rm}:
\begin{cor}
Assume $f:X\to X_{can}$ as in \eqref{semiample} and the regular fiber is not biholomorphic to a finite quotient of a torus. Then for an arbitrary $U\subset\subset X\setminus f^{-1}(V)$ and an arbitrary long-time solution $\overline\omega=\overline\omega(t)$ to the unnormalized K\"ahler-Ricci flow $\partial_t\overline\omega=-Ric(\overline\omega)$, there is a constant $A\ge1$ such that for any $t\ge0$,
\begin{equation}
A^{-1}\le\sup_{U}|Rm(\overline\omega(t))|_{\overline\omega(t)}\le A.
\end{equation}
\end{cor}

\begin{rem}
A type \textrm{III} solution to the K\"ahler-Ricci flow \eqref{nkrf} can be further divided into two different types \cite{Ha93}. Precisely, given a  long time solution $\omega(t)$ to the K\"ahler-Ricci flow \eqref{nkrf} on $X$ and a subset $K\subset X$. Assume $\omega(t)$ is of type \textrm{III} on $K$. We say $\omega(t)$ is of type \textrm{III}(a) on $K$ if for any open neighborhood $U$ of $K$ such that $\limsup_{t\to\infty}\sup_U|Rm(\omega(t))|_{\omega(t)}>0$, and of type \textrm{III}(b) if there is an open neighborhood $U$ of $K$ such that $\limsup_{t\to\infty}\sup_U|Rm(\omega(t))|_{\omega(t)}=0$. Assume $f:X\to X_{can}$ as in \eqref{semiample} and the regular fiber is biholomorphic to a finite quotient of a torus. It is easy to see that the K\"ahler-Ricci flow won't be of type \textrm{III}(b) at any given point. In fact, if there exist some $x\in X$ and some open neighborhood $U$ of $x$ such that $\sup_{U}|Rm(\omega(t))|_{\omega(t)}\to0$, then we must have, for sufficiently large $t$, $\partial_t\omega(t)\le-\frac12\omega(t)$ and so $\omega(t)\le Ce^{-\frac12t}\cdot\omega_0$ on $U$, which, combining the fact that $f(U)\cap(X_{can}\setminus V)$ is an open non-empty set, contradicts to Song-Tian's result that $\omega(t)\ge C^{-1}f^*\chi$ on $X\times[0,\infty)$. Therefore, we have some constant $C=C(U)$ such that for any $t\ge0$, $\sup_{U}|Rm(\omega(t))|_{\omega(t)}\ge C^{-1}$.
Together with our Theorem \ref{thm_Rm}, the infinite-time singularity type of the K\"ahler-Ricci flow at any regular point is understood completely.
\end{rem}

\par Our proofs are achieved by maximum principle arguments and do not involve any convergence results of the K\"ahler-Ricci flow. An important ingredient is the existence of good cut-off functions that are defined locally around each singular/regular fiber.

The structure of this paper is as follows: in Section \ref{sect:cutoff} we constructe some good cut-off functions that will be used to localize our curvature estimates near each fiber. In Sections \ref{sect:Shi} and \ref{sect:Calabi}, we use the cut-off functions constructed to derive local Shi's and Calabi's estimates, then in Section \ref{Rm.Bd} we give the proof of Theorem \ref{tech.thm} by estimating the Riemann curvature locally near a fiber. In Section \ref{sect:MetricEqv}, we derive an important result about local uniform equivalence between two K\"ahler-Ricci flow solutions when the Riemann curvature of one of them is given to be uniformly bounded locally around a fiber. Proofs of above-mentioned theorems are given at the end of this section. Finally in Section \ref{sect:e^t}, we prove Theorems \ref{thm_Rm}; we also give a curvature blow-up rate estimate in terms of the existence of a K\"ahler metric with semi-negative holomorphic sectional curvature.

\section{Cut-off functions near fibers}
\label{sect:cutoff}
\par Recall results of Song-Tian \cite{ST07,ST12,ST16} that, for any solution $\omega(t)$ there exists a constant $C_0\ge1$ such that on $X\times[0,\infty)$,
\begin{equation}\label{lower_bd}
tr_{\omega(t)}(f^*\chi)\le C_0.
\end{equation}
and
\begin{equation}\label{vol_bd}
C_0^{-1}e^{-(n-k)t}\omega_0^n\le\omega(t)^n\le C_0e^{-(n-k)t}\omega_0^n.
\end{equation}
For each fiber $X_y$ (singular or regular), we are going to construct some nice smooth cutoff functions  on $X$ with compact support containing $X_y$ such that they satisfy some desirable bounds according to the estimates \eqref{lower_bd} and \eqref{vol_bd}.

\begin{lem}
\label{cutoff}
For any fiber $X_y$ and an open neighborhood $U$ of $X_y$ in $X$, there exists an open neighborhood $U'$ of $X_y$ in $X$ with $U'\subset\subset U$ and a smooth function $\phi : X \to [0,1]$ such that $\phi$ is compactly supported on $U$, $\phi \equiv 1$ on $U'$, and for each K\"ahler-Ricci flow $\omega(t)$ there exists $C \ge 1$ depending on the initial metric $\omega_0$ and $U$ with
\[\sup_{X \times [0,\infty)}(|\partial\phi|^2_{\omega(t)} + |\Delta_{\omega(t)}\phi|) \leq C.\] 
\end{lem}

\begin{proof}
As an irreducible normal projective variety, $X_{can}$ is a closed subset of $\mathbb{CP}^N$ with the induced topology. Then by compactness of $X$ and $X_y$ there is a proper open subset $\tilde U$ in $X_{can}$ with $y\in\tilde U$ and $f^{-1}(\tilde U)\subset\subset U$. Since $X_{can}$ is a subvariety of $\mathbb{CP}^N$, for the given $y\in\tilde U\subset X_{can}$ we can fix a \emph{sufficiently small} local chart $(\Omega,w^1,...,w^N)$ in $\mathbb{CP}^N$ centered at $y$ such that $X_{can}\cap\Omega\subset\subset\tilde U$. In particular, $f^{-1}(\Omega)=f^{-1}(X_{can}\cap\Omega)\subset\subset U$. We may assume $\Omega=\{(w^1,...,w^N)\in\mathbb C^N||w^1|^2+...|w^N|^2<1\}$. Set $\Omega':=\{(w^1,...,w^N)\in\mathbb C^N||w^1|^2+...|w^N|^2<\frac{1}{m_1^2}\}$ for a sufficiently large number $m_1>1$ such that $X_{can}\cap\Omega'\subset\subset X_{can}\cap\Omega$, and hence $f^{-1}(\Omega')\subset\subset f^{-1}(\Omega)$. Fix a smooth cutoff function $\psi$ on $\mathbb{CP}^N$ which compactly supports on $\Omega$ and identically equals to $1$ on $\Omega'$. There exists a constant $A_0\ge1$ such that on $\mathbb{CP}^N$,
$$\sqrt{-1}\partial\psi\wedge\bar\partial\psi\le A_0\chi.$$
and
$$-A_0\chi\le\sqrt{-1}\partial\bar\partial\psi\le A_0\chi.$$
Then we define a smooth function $\phi:=f^*\psi$ on $X$, which compactly supports on $f^{-1}(\Omega)$ and identically equals to $1$ on $f^{-1}(\Omega')$. Using the above two inequalities and \eqref{lower_bd} gives
\begin{align}
|\partial\phi|_{\omega(t)}^2&=tr_{\omega(t)}(\sqrt{-1}\partial\phi\wedge\bar\partial\phi)\nonumber\\
&=tr_{\omega(t)}(f^*(\sqrt{-1}\partial\psi\wedge\bar\partial\psi))\nonumber\\
&\le A_0tr_{\omega(t)}(f^*\chi)\nonumber\\
&\le A_0C_0\nonumber,
\end{align}

\begin{align}
\Delta_{\omega(t)}\phi&=tr_{\omega(t)}(\sqrt{-1}\partial\bar\partial\phi)\nonumber\\
&=tr_{\omega(t)}(f^*(\sqrt{-1}\partial\bar\partial\psi))\nonumber\\
&\le A_0tr_{\omega(t)}(f^*\chi)\nonumber\\
&\le A_0C_0\nonumber,
\end{align}
and similarly,
\begin{align}
\Delta_{\omega(t)}\phi\ge-A_0C_0\nonumber.
\end{align}
In conclusion, on $X\times[0,\infty)$,
\begin{align}\label{ineq_5}
|\partial\phi|_{\omega(t)}^2\ ,\ |\Delta_{\omega(t)}\phi|\le A_0C_0
\end{align}
as desired. 
\end{proof}

\begin{rem}
We remark that when $y$ is a smooth point in $X_{can}$, we can easily choose a cutoff function by using local chart in $X_{can}$ around $y$. However, when $y$ is a singular point in $X_{can}$, $X_{can}$ is no longer a smooth manifold near $y$ and hence it may be unclear how to find a function on $X_{can}$ which is ``smooth" near $y$. A key point in our above lemma is that, using the ambient manifold $\mathbb{CP}^N$, we can still construct good cutoff function near a singular fiber, which is crucial for the later discussions, as the curvature behaviors near singular fibers will be our main interest.
\end{rem}

\section{Local Shi's derivative estimates near fibers}
\label{sect:Shi}

As a preparation for the next section, we prove a local Shi's estimate by modifying \cite[Section 4]{ShW} on an open neighborhood of a fiber.

\begin{prop}\label{Shi_est}
Let $f:X\to X_{can}$ be the map in \eqref{semiample}. Fix an arbitrary $y\in V$. Assume there exist a solution $\tilde\omega(t)$ of \eqref{nkrf}, an open neighborhood $U$ of $X_y$ in $X$, and an increasing differentiable function $\tau(t) : [0, \infty) \to [1, \infty)$ such that
\begin{equation}\label{type_iii.1.1}
\sup_{U}|Rm(\tilde\omega(t))|_{\tilde\omega(t)}\le\tau(t).
\end{equation}
Then, we have an open neighborhood $U'$ of $X_y$ in $X$ with $U'\subset\subset U$ and a constant $A\ge1$ such that
\begin{equation}\label{Shi}
|\nabla Rm(\tilde\omega)|_{\tilde\omega}\le A\cdot \tau^{\frac32}
\end{equation}
on $U'\times[0,\infty)$. Here $\nabla$ denote the \emph{real} covariant derivative with respect to $\tilde\omega$.
\end{prop}

\begin{proof}
Let $\phi$ be the cutoff function obtained in Lemma \ref{cutoff} and $U' \subset\subset U$ be an open neighborhood of $X_y$ such that $\phi \equiv 1$ on $U'$. We let $A_1 \geq 1$ be a constant such that on $X\times[0,\infty)$,
\begin{align}
|\partial\phi|_{\tilde\omega(t)}^2,\ |\Delta_{\tilde\omega(t)}\phi|\le A_1\nonumber.
\end{align}
By Hamilton there is a constant $C\ge1$ depending only on dimension $n$ such that on $X\times[0,\infty)$,
\begin{equation}
\label{Rm_evolution}
(\partial_t-\Delta_{\tilde\omega})|Rm(\tilde\omega)|^2_{\tilde\omega}\le-|\nabla Rm(\tilde\omega)|^2_{\tilde\omega}+C|Rm(\tilde\omega)|^3_{\tilde\omega}.
\end{equation}
Then, combining with \eqref{type_iii.1.1} and $\frac{d}{dt}\tau\ge0$, we have on $U\times[0,\infty)$,
\begin{equation}\label{ineq_22'}
(\partial_t-\Delta_{\tilde\omega})(\tau^{-2}|Rm(\tilde\omega)|^2_{\tilde\omega})\le-\tau^{-2}|\nabla Rm(\tilde\omega)|^2_{\tilde\omega}+C\tau.
\end{equation}
On the other hand, 
\begin{align}
(\partial_t-\Delta_{\tilde\omega})|\nabla Rm(\tilde\omega)|^2_{\tilde\omega}&\le-|\nabla^2 Rm(\tilde\omega)|^2_{\tilde\omega}+C|Rm(\tilde\omega)|_{\tilde\omega}\cdot|\nabla Rm(\tilde\omega)|^2_{\tilde\omega}\nonumber\\
&\le -|\nabla^2 Rm(\tilde\omega)|^2_{\tilde\omega}+C\tau|\nabla Rm(\tilde\omega)|^2_{\tilde\omega}
\end{align}
Thus
\begin{align}\label{ineq_23'}
&(\partial_t-\Delta_{\tilde\omega})(\phi^2|\nabla Rm(\tilde\omega)|^2_{\tilde\omega})\nonumber\\
&=\phi^2(\partial_t-\Delta_{\tilde\omega})|\nabla Rm(\tilde\omega)|^2_{\tilde\omega}-\Delta_{\tilde\omega}(\phi^2)|\nabla Rm(\tilde\omega)|^2_{\tilde\omega}-2Re(\partial(\phi^2)\cdot\bar\partial|\nabla Rm(\tilde\omega)|^2_{\tilde\omega})\nonumber\\
&\le-\phi^2|\nabla^2 Rm(\tilde\omega)|^2_{\tilde\omega}+C\tau|\nabla Rm(\tilde\omega)|^2_{\tilde\omega}-2Re(\partial(\phi^2)\cdot\bar\partial|\nabla Rm(\tilde\omega)|^2_{\tilde\omega})
\end{align}
By Cauchy-Schwarz inequality,
\begin{align}
-2Re(\partial(\phi^2)\cdot\bar\partial|\nabla Rm(\tilde\omega)|^2_{\tilde\omega})&\le\phi^2|\nabla^2 Rm(\tilde\omega)|^2_{\tilde\omega}+4|\partial\phi|^2\cdot |\nabla Rm(\tilde\omega)|^2_{\tilde\omega}\nonumber\\
&\le\phi^2|\nabla^2 Rm(\tilde\omega)|^2_{\tilde\omega}+4A_1\cdot |\nabla Rm(\tilde\omega)|^2_{\tilde\omega}\nonumber,
\end{align}
and by putting it into \eqref{ineq_23'} gives, for some constant $C\ge1$, 
\begin{align}\label{ineq_23''}
(\partial_t-\Delta_{\tilde\omega})(\phi^2|\nabla Rm(\tilde\omega)|^2_{\tilde\omega})\le C\tau|\nabla Rm(\tilde\omega)|^2_{\tilde\omega}
\end{align}
and so
\begin{align}\label{ineq_23'''}
(\partial_t-\Delta_{\tilde\omega})(\tau^{-3}\phi^2|\nabla Rm(\tilde\omega)|^2_{\tilde\omega})\le C\tau^{-2}|\nabla Rm(\tilde\omega)|^2_{\tilde\omega}
\end{align}
on $U\times[0,\infty)$.
By combining \eqref{ineq_22'} and \eqref{ineq_23'''}, we can choose a constant $D\ge1$ such that there holds on $U\times[0,\infty)$ that
\begin{align}\label{ineq_23''''}
(\partial_t-\Delta_{\tilde\omega})(\tau^{-3}\phi^2|\nabla Rm(\tilde\omega)|^2_{\tilde\omega}+D\tau^{-2}| Rm(\tilde\omega)|^2_{\tilde\omega})\le -\tau^{-2}|\nabla Rm(\tilde\omega)|^2_{\tilde\omega}+C\tau.
\end{align}
Note that by our assumption \eqref{type_iii.1.1} on $|Rm|$ and that $\phi = 0$ on $\partial U$, we know that $\tau^{-3}\phi^2|\nabla Rm(\tilde\omega)|^2_{\tilde\omega}+D\tau^{-2}| Rm(\tilde\omega)|^2_{\tilde\omega} \leq D$ on $\partial U$. By applying the maximum principle arguments on $\bar U\times[0,\infty)$ we know $\tau^{-3}\phi^2|\nabla Rm(\tilde\omega)|^2_{\tilde\omega}+D\tau^{-2}| Rm(\tilde\omega)|^2_{\tilde\omega}$ is uniformly bounded on $U\times[0,\infty)$ and by $\phi \equiv 1$ on $U'$ we proved that $\tau^{-3}|\nabla Rm(\tilde\omega)|^2_{\tilde\omega}$ is uniformly bounded
on $U'\times[0,\infty)$. In other words, there is a constant $C\ge1$ such that on $U'\times[0,\infty)$,
\begin{equation}
|\nabla Rm(\tilde\omega)|_{\tilde\omega}\le C\cdot \tau^{\frac32}\nonumber,
\end{equation}
completing the proof.
\end{proof}

\section{Local Calabi's $C^3$-estimate near fibers}
\label{sect:Calabi}
As in \cite{Yau,C,PSS,ShW,ShW2}, we define a tensor $\Psi=(\Psi_{ij}^k)$ by $\Psi_{ij}^k:=\Gamma_{ij}^k-\tilde\Gamma^k_{ij}$, where $\Gamma_{ij}^k$ (resp. $\tilde\Gamma_{ij}^k$) is the Christoffel symbols of $\omega(t)$ (resp. $\tilde\omega(t)$), and $S=|\Psi|_\omega^2$.  Next we derive an upper estimate of $S$ in the setting as in Theorem \ref{tech.thm}. Note that similar local estimates appear in Sherman-Weinkove's works \cite{ShW, ShW2} in which the K\"ahler-Ricci flow (and more generally Chern-Ricci flow) solution is assumed to be locally uniformly equivalent to a \emph{fixed} metric on a ball. In our setting, the two metrics are both evolving and we allow the eigenvalues of $\tilde{\omega}(t)^{-1} \omega(t)$ to blow-up as $t \to \infty$. The \emph{good} cut-off functions constructed in Lemma \ref{cutoff} are essential in our proof to tackle these issues.

\begin{prop}\label{Calabi_est}
Assume the same setting as in Theorem \ref{tech.thm}. Let $U'$ be an open neighborhood of $X_y$ satisfying \eqref{Shi} in Proposition \ref{Shi_est}. Then we have an open neighborhood $U''$ of $X_y$ in $X$ with $U''\subset\subset U'$ and a constant $A\ge1$ such that
\begin{equation}\label{Calabi}
S\le A\sigma^4\tau
\end{equation}
on $U''\times[0,\infty)$.
\end{prop}

\begin{proof}
Recall the evolution of $tr_{\omega}\tilde\omega$:
\begin{equation}
(\partial_t-\Delta_{\omega})tr_{\omega}\tilde\omega=-tr_{\omega}(Ric(\tilde\omega))+g^{\bar ji}g^{\bar qp}\tilde R_{i\bar jp\bar q}-g^{\bar ji}g^{\bar qp}\tilde g^{\bar ba}\nabla_{i}\tilde g_{p\bar b}\nabla_{\bar j}\tilde g_{a\bar q}\nonumber.
\end{equation}
Then by \eqref{tau} and \eqref{sigma} we get, on $U\times[0,\infty)$,

\begin{equation}\label{ineq_14.0}
(\partial_t-\Delta_{\omega})tr_{\omega}\tilde\omega\le 2n\tau\sigma^2-\sigma^{-1}S,
\end{equation}
and so
\begin{equation}\label{ineq_14.1}
(\partial_t-\Delta_{\omega})(\sigma^{-1}tr_{\omega}\tilde\omega)\le 2n\tau\sigma-\sigma^{-2}S,
\end{equation}
Recall \cite[(3.17)]{Zys17}:
\begin{equation}\label{C36.1}
(\partial_t-\Delta_\omega)S=S-|\nabla\Psi|_{\omega}^2-|\overline{\nabla}\Psi|_\omega^2+2Re(g^{\bar ji}g^{\bar qp}g_{k\bar l}(\tilde\nabla_{i}\tilde R_{p}^{\ \,k}-\nabla^{\bar b}\tilde R_{i\bar bp}^{\ \ \ \,k})\overline{\Psi_{jq}^l})
\end{equation}
and 
$$\nabla^{\bar b}\tilde R_{i\bar bp}^{\ \ \ \,k}=\Psi*Rm(\tilde\omega)+g^{\bar ba}\tilde\nabla_a\tilde R_{i\bar bp}^{\ \ \ \,k}$$
By \eqref{tau} and \eqref{Shi},
$$|\tilde\nabla_{i}\tilde R_{p}^{\ \,k}|_{\omega}\le\sigma^{\frac32}|\tilde\nabla_{i}\tilde R_{p}^{\ \,k}|_{\tilde\omega}\le C\sigma^{\frac32}\tau^{\frac32},$$
$$|\Psi*Rm(\tilde\omega)|_{\omega}\le|\Psi|_{\omega}|Rm(\tilde\omega)|_{\omega}\le C\sigma^{2}\tau|\Psi|_{\omega}$$
and 
$$|g^{\bar ba}\tilde\nabla_a\tilde R_{i\bar bp}^{\ \ \ \,k}|_{\omega}\le\sigma|\tilde\nabla^{\bar b}\tilde R_{i\bar bp}^{\ \ \ \,k}|_{\omega}\le\sigma^{\frac52}|\tilde\nabla^{\bar b}\tilde R_{i\bar bp}^{\ \ \ \,k}|_{\tilde\omega}\le C\sigma^{\frac52}\tau^{\frac32},$$
implying
$$2Re(g^{\bar ji}g^{\bar qp}g_{k\bar l}(\tilde\nabla_{i}\tilde R_{p}^{\ \,k}-\nabla^{\bar b}\tilde R_{i\bar bp}^{\ \ \ \,k})\overline{\Psi_{jq}^l})\le C\sigma^{2}\tau S+C\sigma^{\frac52}\tau^{\frac32}\sqrt{S}.$$
Therefore, we have
\begin{equation}\label{C36.1'}
(\partial_t-\Delta_\omega)S\le C\sigma^{2}\tau S+C\sigma^{\frac52}\tau^{\frac32}\sqrt{S}-|\nabla\Psi|_{\omega}^2-|\overline{\nabla}\Psi|_\omega^2.
\end{equation}
and so
\begin{equation}\label{C36.1''}
(\partial_t-\Delta_\omega)\left(\sigma^{-4}\tau^{-1}S\right)\le C\sigma^{-2}S+C\sigma^{-\frac32}\tau^{\frac12}\sqrt{S}-\sigma^{-4}\tau^{-1}(|\nabla\Psi|_{\omega}^2+|\overline{\nabla}\Psi|_\omega^2),
\end{equation}
which holds on $U'\times[0,\infty)$.

\par Now by Lemma \ref{cutoff} we choose an open neighborhood $U''$ of $X_y$ with $U''\subset\subset U'$ and fix a cutoff function $\phi$ on $X$, which is compactly supported on $U'$, identically equals to $1$ on $U''$, and satisfies, for a constant $C\ge1$, on $X\times[0,\infty)$,
\begin{align}\label{ineq_16}
|\partial\phi|_{\omega(t)}^2+|\Delta_{\omega(t)}\phi|\le C.
\end{align}
Compute:
\begin{align}
&(\partial_t-\Delta_\omega)(\phi^2\sigma^{-4}\tau^{-1}S)\nonumber\\
&=\phi^2(\partial_t-\Delta_\omega)(\sigma^{-4}\tau^{-1}S)-(\Delta_\omega\phi^2)\sigma^{-4}\tau^{-1}S-2\sigma^{-4}\tau^{-1}Re(\partial(\phi^2)\cdot\bar\partial S)\nonumber\\
&\le C\phi^2\sigma^{-2}S+C\phi^2\sigma^{-\frac32}\tau^{\frac12}\sqrt{S}-\phi^2\sigma^{-4}\tau^{-1}(|\nabla\Psi|_{\omega}^2+|\overline{\nabla}\Psi|_\omega^2)\\
& \hskip 0.5cm +C\sigma^{-4}\tau^{-1}S-2\sigma^{-4}\tau^{-1}Re(\partial(\phi^2)\cdot\bar\partial S)\nonumber\\
&\le C\sigma^{-2}S+C\sigma^{-\frac32}\tau^{\frac12}\sqrt{S}-\phi^2\sigma^{-4}\tau^{-1}(|\nabla\Psi|_{\omega}^2+|\overline{\nabla}\Psi|_\omega^2)\\
& \hskip 0.5cm -2\sigma^{-4}\tau^{-1}Re(\partial(\phi^2)\cdot\bar\partial S)\nonumber.
\end{align}
Note that 
$$-2\sigma^{-4}\tau^{-1}Re(\partial(\phi^2)\cdot\bar\partial S)\le \sigma^{-4}\tau^{-1}\phi^2(|\nabla\Psi|_{\omega}^2+|\overline{\nabla}\Psi|_\omega^2)+C\sigma^{-4}\tau^{-1}S,$$
putting which into the above inequality gives
\begin{align}\label{ineq_17.1}
(\partial_t-\Delta_\omega)(\phi^2\sigma^{-4}\tau^{-1}S)\le C\sigma^{-2}S+C\sigma^{-\frac32}\tau^{\frac12}\sqrt{S}.
\end{align}
Therefore, by setting $Q:=\phi^2\sigma^{-4}\tau^{-1}S+C\sigma^{-1}tr_{\omega}\tilde\omega$ for a sufficiently large constant $C$, then by \eqref{ineq_14.1} and \eqref{ineq_17.1} there holds on $U'\times[0,\infty)$ that
\begin{align}\label{ineq_19}
(\partial_t-\Delta_\omega)Q\le-\sigma^{-2}S+C\sigma^{-\frac32}\tau^{\frac12}\sqrt{S}+C\sigma\tau.
\end{align}
Assume $(\bar x,\bar t)$ is a maximal point of $Q$ on $\overline{U'}\times[0,T]$. If $\bar{x} \in \partial U'$, then by that $\phi = 0$ on $\partial U'$ and $tr_\omega \tilde{\omega}$ being uniformly bounded, we would have $Q(\bar{x}, \bar{t}) \leq C$ already. We may assume $\bar x\in U'$, $\bar t>0$ and $S\ge B\sigma\tau$ at $(\bar x,\bar t)$ for a sufficiently large constant $B$ so that $C\sigma^{-3/2}\tau^{1/2}\sqrt{S} \leq \frac{1}{2}\sigma^{-2}S$. Then by the maximum principle we have, at $(\bar x,\bar t)$,
\begin{align}
\sigma^{-2}S&\le C\sigma^{-\frac32}\tau^{\frac12}\sqrt{S}+C\sigma\tau\nonumber\\
&\le \frac{1}{2}\sigma^{-2}S + C\sigma\tau.
\end{align}
By rearrangement, we get
\begin{align}
S\le \tilde{C}\sigma^3\tau\nonumber
\end{align}
for some large constant $\tilde{C}$, proving that $Q(\bar x,\bar t)$ is uniformly bounded on $U'\times[0,\infty)$, and so $\sigma^{-4}\tau^{-1}S$ is uniformly bounded on $U''\times[0,\infty)$. The proof is then completed.
\end{proof}

\section{Curvature estimates near fibers and proof of Theorem \ref{tech.thm}}
\label{Rm.Bd}
In this section, we give the proof of our major results (Theorem \ref{tech.thm}). The proof is modified from \cite[Section 3]{ShW}. Since our case involves two different solutions (both are degenerate at time-infinity) and some possibly unbounded quantities, which is slightly different from the setting in \cite[Section 3]{ShW}, we present the details here for convenience of readers.
\begin{proof}[Proof of Theorem \ref{tech.thm}]
\par Firstly, by combining \eqref{C36.1'} and \eqref{Calabi}, we have on $U''\times[0,\infty)$,
\begin{equation}\label{S_ineq}
(\partial_t-\Delta_\omega)S\le C\sigma^{6}\tau^2-|\nabla\Psi|_{\omega}^2-|\overline{\nabla}\Psi|_\omega^2.
\end{equation}
By \eqref{tau} and \eqref{sigma} we also have
\begin{align}\label{Psi}
|\overline\nabla\Psi|_\omega^2=|\tilde R_{i\bar bp}^{\ \ \ \,k}-R_{i\bar bp}^{\ \ \ \,k}|_\omega^2\ge\frac{1}{2}|Rm(\omega)|_\omega^2-C|Rm(\tilde\omega)|^2_{\omega}\ge\frac{1}{2}|Rm(\omega)|_\omega^2-C\sigma^4\tau^2,
\end{align}
putting which into \eqref{S_ineq} concludes
\begin{equation}\label{S_ineq.1}
(\partial_t-\Delta_\omega)S\le C\sigma^{6}\tau^2-\frac{1}{2}|Rm(\omega)|_\omega^2,
\end{equation}
Now \eqref{S_ineq} and \eqref{S_ineq.1} imply
\begin{equation}\label{S_ineq.2}
(\partial_t-\Delta_\omega)(\sigma^{-4}\tau^{-1}S)\le C\sigma^{2}\tau-\sigma^{-4}\tau^{-1}(|\nabla\Psi|_{\omega}^2+|\overline{\nabla}\Psi|_\omega^2)
\end{equation}
and
\begin{equation}\label{S_ineq.3}
(\partial_t-\Delta_\omega)(\sigma^{-4}\tau^{-1}S)\le C\sigma^{2}\tau-\frac{1}{2}\sigma^{-4}\tau^{-1}|Rm(\omega)|_\omega^2
\end{equation}
respectively. Also recall
\begin{equation}\label{Rm_ineq}
(\partial_t-\Delta_\omega)|Rm(\omega)|_\omega^2\le-|\nabla Rm(\omega)|_\omega^2-|\overline\nabla Rm(\omega)|_\omega^2+C|Rm(\omega)|_\omega^3.
\end{equation}

Now we set $\tilde S:=\sigma^{-4}\tau^{-1}S$, which is a smooth bounded function on $U''\times[0,\infty)$. Again by Lemma \ref{cutoff} we choose an open neighborhood $U'''$ of $X_y$ with $U'''\subset\subset U''$ and fix a cut-off function $\phi$ on $X$, which is compactly supported on $U''$, identically equals to $1$ on $U'''$, and satisfies, for a constant $C\ge1$, on $X\times[0,\infty)$,
\begin{align}\label{cutoff'''}
|\partial\phi|_{\omega(t)}^2+|\Delta_{\omega(t)}\phi|\le C\nonumber.
\end{align}
Let $B$ be a sufficiently large constant so that $B -\tilde{S} > \frac{B}{2}$. We then modify the direct computations in \cite{ShW} and get:
\begin{align}
\nonumber & (\partial_t-\Delta_\omega)\left(\phi^2\frac{|Rm(\omega)|_\omega^2}{B-\tilde S}\right)\\
=&-(\Delta\phi^2)\frac{|Rm(\omega)|_\omega^2}{B-\tilde S}+\phi^2\frac{(\partial_t-\Delta_\omega)|Rm(\omega)|_\omega^2}{B-\tilde S}+\phi^2\frac{(\partial_t-\Delta_\omega)\tilde S}{(B-\tilde S)^2}|Rm(\omega)|_\omega^2\nonumber\\
&-2\phi^2\frac{|\partial\tilde S|_\omega^2|Rm(\omega)|_\omega^2}{(B-\tilde S)^3}-4Re\frac{\phi\cdot\partial\phi\cdot\bar\partial|Rm(\omega)|_\omega^2}{B-\tilde S}\nonumber\\
&-4Re\frac{\phi\partial\phi\cdot\bar\partial\tilde S}{(B-\tilde S)^2}|Rm(\omega)|_\omega^2-2Re\frac{\phi^2\cdot\partial|Rm(\omega)|_\omega^2\cdot\bar\partial\tilde S}{(B-\tilde S)^2}.
\end{align}
Using \eqref{S_ineq.2}, \eqref{Rm_ineq} and Cauchy-Schwarz inequality, we have
\begin{align}
& (B-\tilde S)^{2}(\partial_t-\Delta_\omega)\left(\phi^2\frac{|Rm(\omega)|_\omega^2}{B-\tilde S}\right)\\
\le &-(\Delta\phi^2)(B-\tilde S)|Rm(\omega)|_\omega^2\nonumber\\
&+\phi^2(B-\tilde S)(C|Rm(\omega)|_\omega^3-|\nabla Rm(\omega)|_\omega^2-|\overline\nabla Rm(\omega)|_\omega^2)\nonumber\\
&+\phi^2(C\sigma^{2}\tau-\sigma^{-4}\tau^{-1}|\nabla\Psi|_{\omega}^2-\sigma^{-4}\tau^{-1}|\overline{\nabla}\Psi|_\omega^2)|Rm(\omega)|_\omega^2\nonumber\\
&-\frac{2}{B-\tilde S}\phi^2|\partial\tilde S|_\omega^2|Rm(\omega)|_\omega^2+16|\partial\phi|_\omega^2(B-\tilde S)|Rm(\omega)|_\omega^2\nonumber\\
&+\frac12\phi^2(B-\tilde S)|\nabla Rm(\omega)|_\omega^2+\frac12\phi^2(B-\tilde S)|\overline\nabla Rm(\omega)|_\omega^2\nonumber\\
&+\frac{1}{B-\tilde S}\phi^2|\partial\tilde S|_\omega^2|Rm(\omega)|_\omega^2
+4|\partial\phi|_\omega^2(B-\tilde S)|Rm(\omega)|_\omega^2\nonumber\\
&+\frac{4}{B-\tilde S}\phi^2|\partial\tilde S|_\omega^2|Rm(\omega)|_\omega^2\nonumber\\
&+\frac12\phi^2(B-\tilde S)|\nabla Rm(\omega)|_\omega^2+\frac12\phi^2(B-\tilde S)|\overline\nabla Rm(\omega)|_\omega^2\nonumber.
\end{align}

Label the above terms $(1), (2), \ldots, (16)$. Observe that
\begin{equation}
(1)+(5)+(9)+(13)\le C\sigma^{2}\tau|Rm(\omega)|_\omega^2,
\end{equation}
\begin{equation}
(3)+(4)+(10)+(11)+(15)+(16)=0
\end{equation}
and
\begin{equation}
(8)+(12)+(14)=\frac{3}{B-\tilde S}\phi^2|\partial\tilde S|_\omega^2|Rm(\omega)|_\omega^2
\end{equation}

Then we have
\begin{align}
& (B-\tilde S)^{2}(\partial_t-\Delta_\omega)\left(\phi^2\frac{|Rm(\omega)|_\omega^2}{B-\tilde S}\right)\\
\le & \, C\sigma^{2}\tau|Rm(\omega)|_\omega^2\nonumber\\
&+\frac{3}{B-\tilde S}\phi^2|\partial\tilde S|_\omega^2|Rm(\omega)|_\omega^2\nonumber\\
&+\phi^2C(B-\tilde S)|Rm(\omega)|_\omega^3\nonumber\\
&-\phi^2\sigma^{-4}\tau^{-1}|\nabla\Psi|_{\omega}^2|Rm(\omega)|_\omega^2-\phi^2\sigma^{-4}\tau^{-1}|\overline{\nabla}\Psi|_\omega^2|Rm(\omega)|_\omega^2.
\end{align}
Label the above terms $(T_1), (T_2), \ldots, (T_5)$. Using \eqref{Psi} we have
\begin{align}
(T_3)+\frac12(T_5)&\le\phi^2C(B-\tilde S)|Rm(\omega)|_\omega^3-\phi^2\sigma^{-4}\tau^{-1}\left(\frac14|Rm(\omega)|_\omega^2-C\sigma^4\tau^2\right)|Rm(\omega)|_\omega^2\nonumber\\
&\leq\phi^2|Rm(\omega)|_\omega^3\left(BC-\frac14\sigma^{-4}\tau^{-1}|Rm(\omega)|_\omega\right)+C\tau|Rm(\omega)|_\omega^2.
\end{align}
Observing that by Cauchy-Schwarz inequality and \eqref{Calabi}, we have
$$|\partial\tilde S|_\omega^2=\sigma^{-8}\tau^{-2}|\partial S|_\omega^2\le 4\sigma^{-8}\tau^{-2}S(|\nabla\Psi|_\omega^2+|\overline\nabla\Psi|_\omega^2)\le C\sigma^{-4}\tau^{-1}(|\nabla\Psi|_\omega^2+|\overline\nabla\Psi|_\omega^2),$$
and so for sufficiently large constant $B$,
\begin{align}
(T_2)+\frac12(T_4)+\frac12(T_5)\le\left(\frac{3C}{B-\tilde S}-\frac{1}{2}\right)\phi^2\sigma^{-4}\tau^{-1}(|\nabla\Psi|_\omega^2+|\overline\nabla\Psi|_\omega^2)|Rm(\omega)|_\omega^2\le0.
\end{align}
Then we arrive at
\begin{align}
& (B-\tilde S)^{2}(\partial_t-\Delta_\omega)\left(\phi^2\frac{|Rm(\omega)|_\omega^2}{B-\tilde S}\right)\nonumber\\
& \le C\sigma^{2}\tau|Rm(\omega)|_\omega^2+\phi^2|Rm(\omega)|_\omega^3\left(BC-\frac14\sigma^{-4}\tau^{-1}|Rm(\omega)|_\omega\right).
\end{align}
Note we have discarded a term $\frac12(T_4)$.

This shows
\begin{align}\label{ev_Rm}
& (\partial_t-\Delta_\omega)\left(\phi^2\sigma^{-8}\tau^{-2}\frac{|Rm(\omega)|_\omega^2}{B-\tilde S}\right)\nonumber\\
\le & \, C\sigma^{-6}\tau^{-1}|Rm(\omega)|_\omega^2\nonumber\\
&+\sigma^{-8}\tau^{-2}(B-\tilde S)^{-2}\phi^2|Rm(\omega)|_\omega^3\left(BC-\frac14\sigma^{-4}\tau^{-1}|Rm(\omega)|_\omega\right),
\end{align}
Now consider $Q:=\phi^2\sigma^{-8}\tau^{-2}\frac{|Rm(\omega)|_\omega^2}{B-\tilde S}+C\tilde S$ for sufficiently large constant $C$. By combining \eqref{S_ineq.3} and \eqref{ev_Rm} we get, on $U''\times[0,\infty)$,
\begin{align}\label{ev_Rm.1}
(\partial_t-\Delta_\omega)Q\le& -\sigma^{-4}\tau^{-1}|Rm(\omega)|_\omega^2+C\sigma^2\tau\nonumber\\
&+\sigma^{-8}\tau^{-2}(B-\tilde S)^{-2}\phi^2|Rm(\omega)|_\omega^3\left(BC-\frac14\sigma^{-4}\tau^{-1}|Rm(\omega)|_\omega\right).
\end{align}
Suppose $(\bar x,\bar t)$ is a maximum point of $Q$ on $\overline{U''}\times[0,T]$. If $\bar{x} \in \partial U''$, then at $(\bar{x}, \bar{t})$ we have $Q$ being uniformly bounded since $\phi = 0$ on $\partial U''$ and $\tilde{S}$ is uniformly bounded. We may assume without loss of generality that $\phi(\bar x)>0$ (and hence $\bar x\in U'$), $\bar t>0$ and $|Rm(\omega)|^2_\omega\ge C\sigma^6\tau^2$ at $(\bar x,\bar t)$. Then by applying the maximum principle on \eqref{ev_Rm.1} we have at $(\bar{x},\bar{t})$
\[BC - \frac{1}{4}\sigma^{-4}\tau^{-1}|Rm(\omega)|_{\omega} \ge 0\]
and hence $|Rm(\omega)|_{\omega} (\bar{x},\bar{t}) \leq C\sigma^4\tau$ and $Q(\bar x,\bar t)\le C$. In conclusion, $Q \leq C$ on $U''\times[0,\infty)$, and hence $|Rm(\omega)|_\omega\le C\sigma^{4}\tau$ on $U'''\times[0,\infty)$, completing the proof.

\end{proof}

\section{Metric equivalences, and Proofs of Theorems \ref{blowuprate} and \ref{ind.1}}
\label{sect:MetricEqv}

In this section, we prove an important estimate between the local upper bound near a fiber $X_y$ of the Riemann curvature of a particular K\"ahler-Ricci flow solution and the local bound near $X_y$ on the eigenvalues of any other K\"ahler-Ricci flow solution.

\begin{prop}\label{equivalence}
Let $f:X\to X_{can}$ be the map in \eqref{semiample}. Fix an arbitrary $y\in X_{can}$. Assume there exist a solution $\tilde\omega(t)$, an open neighborhood $U$ of $X_y$ in $X$, and an increasing differentiable function $\tau(t) : [0, \infty) \to [1, \infty)$ such that
\begin{equation}\label{type_iii.1}
\sup_{U}|Rm(\tilde\omega(t))|_{\tilde\omega(t)}\le\tau(t).
\end{equation}
Then, we have an open neighborhood $U'$ of $X_y$ in $X$ with $U'\subset\subset U$ such that for any solution $\omega(t)$ there exists a constant $A\ge1$ satisfying
\begin{equation}\label{metric_equivalence}
e^{-A\cdot\tau(t)}\tilde\omega(t)\le\omega(t)\le e^{A\cdot\tau(t)}\tilde\omega(t)
\end{equation}
on $U'\times[0,\infty)$.
\end{prop}

\begin{proof}
By direct computation (see e.g. \cite[Section 3]{Zys17}) we have
\begin{equation}
(\partial_t-\Delta_{\omega})tr_{\omega}\tilde\omega=-tr_{\omega}(Ric(\tilde\omega))+g^{\bar ji}g^{\bar qp}\tilde R_{i\bar jp\bar q}-g^{\bar ji}g^{\bar qp}\tilde g^{\bar ba}\nabla_{i}\tilde g_{p\bar b}\nabla_{\bar j}\tilde g_{a\bar q}\nonumber.
\end{equation}
Combining \eqref{type_iii.1}, we get
\begin{equation}\label{ineq_5.0}
(\partial_t-\Delta_{\omega})tr_{\omega}\tilde\omega\le n\tau\cdot tr_{\omega}\tilde\omega+\tau(tr_{\omega}\tilde\omega)^2-g^{\bar ji}g^{\bar qp}\tilde g^{\bar ba}\nabla_{i}\tilde g_{p\bar b}\nabla_{\bar j}\tilde g_{a\bar q}.
\end{equation}
and hence
\begin{equation}\label{ineq_5.0'}
(\partial_t-\Delta_{\omega})\log tr_{\omega}\tilde\omega\le \tau\cdot tr_{\omega}\tilde\omega+n\tau,
\end{equation}
which by $\frac{d}{dt}\tau\ge0$ implies, at any point with $tr_\omega\tilde\omega\ge1$,
\begin{equation}\label{ineq_5.0''}
(\partial_t-\Delta_{\omega})\left(\tau^{-1}\log tr_{\omega}\tilde\omega\right)\le tr_{\omega}\tilde\omega+n,
\end{equation}

As in \cite[(3.6)]{Zys17}, there is a smooth real function $u=u(t)$ on $X\times[0,\infty)$ and a constant $C\ge1$ such that 
\begin{equation}\label{ineq_5.1}
\sup_{X\times[0,\infty)}|u|+|\partial_tu|\le C
\end{equation}
and 
\begin{equation}\label{ineq_6}
(\partial_t-\Delta_{\omega})(\tau^{-1}\log tr_{\omega}\tilde\omega+u(t))\le-tr_{\omega}\tilde\omega+C
\end{equation}
at any point in $U\times[0,\infty)$ with $tr_{\omega}\tilde\omega\ge1$. Precisely, we let $\varphi(t)$ be the K\"ahler potential of $\omega(t)$ satisfying $\omega(t)=
e^{-t}\omega_0+(1-e^{-t})f^*\chi+\sqrt{-1}\partial\bar\partial\varphi(t)$ and $\varphi(0)=0$; and similarly define $\tilde\varphi(t)$ for $\tilde{\omega}(t)$. Now, if we fix a constant $\tilde A\ge2$ such that $\tilde\omega_0\le\frac{\tilde A}{2}\cdot\omega_0$ on $X$, then we claim that $u=u(t):=2\tilde\varphi(t)-\tilde A\cdot\varphi(t)$ would satisfy both \eqref{ineq_5.1} and \eqref{ineq_6}. In fact, the \eqref{ineq_5.1} is exactly the uniform $C^0$-estimates on $X \times [0,\infty)$ of Song-Tian \cite{ST07,ST12,ST16}, while direct computations give
\begin{align}
(\partial_t-\Delta_{\omega})u&=(\partial_t-\Delta_{\omega})(2\tilde\varphi-\tilde A\cdot\varphi)\nonumber\\
&=-2tr_\omega\tilde\omega+e^{-t}tr_\omega(2\tilde\omega_0-\tilde A\cdot\omega_0)+(1-e^{-t})(2-\tilde A)tr_{\omega}f^*\chi+\partial_tu+\tilde An\nonumber\\
&\le-2tr_\omega\tilde\omega+C'\nonumber
\end{align}
on $X\times[0,\infty)$, combining which with \eqref{ineq_5.0''} immediately implies \eqref{ineq_6}.
\par By Lemma \ref{cutoff} we choose an open neighborhood $U'$ of $X_y$ with $U'\subset\subset U$ and fix a cutoff function $\phi$ on $X$, which is compactly supported on $U'$, identically equals to $1$ on $U'$, and satisfies, for a constant $C\ge1$, on $X\times[0,\infty)$,
\begin{align}
|\partial\phi|_{\omega(t)}^2+|\Delta_{\omega(t)}\phi|\le C\nonumber.
\end{align}
To apply the maximum principle, we consider the function $P(t):=\phi\cdot\left(\tau^{-1}\log tr_{\omega}\tilde\omega+u(t)\right)$. We have

\begin{align}\label{ineq_7}
(\partial_t-\Delta_{\omega})P(t)=&\phi\cdot(\partial_t-\Delta_{\omega})(\tau^{-1}\log tr_{\omega}\tilde\omega+u(t))-(\Delta_{\omega}\phi)\cdot(\tau^{-1}\log tr_{\omega}\tilde\omega+u(t))\nonumber\\
&-2Re\left(\partial\phi\cdot\bar\partial(\tau^{-1}\log tr_{\omega}\tilde\omega+u(t))\right)
\end{align}

For any $T>0$, let $(\bar x,\bar t)$ be a maximal point of $P$ on $X\times[0,T]$. We may assume $\bar t>0$, and $\phi(\bar{x}) > 0$ (otherwise $tr_{\omega}(\tilde{\omega}) \leq Ce^{A\tau}$ at the maximum point of $P$ and we are done) and by similar reason assume that $tr_{\omega}\tilde\omega(\bar x,\bar t)\ge e^{2+\tau C_2}$ so that $P(\bar x,\bar t)>0$. Applying the maximum principle in \eqref{ineq_7} and using \eqref{ineq_5}, \eqref{ineq_6}, we have at $(\bar x,\bar t)$,
\begin{align}\label{ineq_8}
0 &\le(\partial_t-\Delta_{\omega})P\nonumber\\
& =\phi\cdot(\partial_t-\Delta_{\omega})(\tau^{-1}\log tr_{\omega}\tilde\omega+u)-(\Delta_{\omega}\phi)\cdot(\tau^{-1}\log tr_{\omega}\tilde\omega+u)\nonumber\\
& \hskip 0.5cm -2Re\left(\partial\phi\cdot\bar\partial(\tau^{-1}\log tr_{\omega}\tilde\omega+u)\right)\nonumber\\
& \le\phi(-tr_{\omega}\tilde\omega+C)+C\cdot(\tau^{-1}\log tr_{\omega}\tilde\omega+u)\nonumber \\
& \hskip 0.5cm -2Re\left(\partial\phi\cdot\bar\partial(\tau^{-1}\log tr_{\omega}\tilde\omega+u)\right).
\end{align}
On the other hand, at $(\bar x,\bar t)$ we have
$$0=\bar \partial P=(\tau^{-1}\log tr_{\omega}\tilde\omega+u)\cdot\bar\partial\phi+\phi\cdot\bar\partial(\tau^{-1}\log tr_{\omega}\tilde\omega+u),$$
or equivalently (note that $\phi(\bar{x}) > 0$), we have
$$\bar\partial(\tau^{-1}\log tr_{\omega}\tilde\omega+u)=-\frac{(\tau^{-1}\log tr_{\omega}\tilde\omega+u)}{\phi}\cdot\bar\partial\phi.$$
This shows
\begin{align}
-Re\left(\partial\phi\cdot\bar\partial(\tau^{-1}\log tr_{\omega}\tilde\omega+u)\right)&=\frac{(\tau^{-1}\log tr_{\omega}\tilde\omega+u)}{\phi}\cdot|\partial\phi|_{\omega(t)}^2\nonumber\\
&\le C\frac{(\tau^{-1}\log tr_{\omega}\tilde\omega+u)}{\phi}\nonumber,
\end{align}
and by putting it into \eqref{ineq_8} we get
\begin{align}
tr_\omega\tilde\omega&\le\frac{C\cdot(\tau^{-1}\log tr_{\omega}\tilde\omega+u)}{\phi}+\frac{C(\tau^{-1}\log tr_{\omega}\tilde\omega+u)}{\phi^2}+C\nonumber\\
&\le\frac{C\log tr_{\omega}\tilde\omega+C}{\phi^2}\nonumber\\
&\le\frac{C\log tr_{\omega}\tilde\omega}{\phi^2}\nonumber
\end{align}
where in the second inequality uses the uniform bound for $u$ in \eqref{ineq_5.1}, $\tau\ge1$ and $\phi\le1$, and in the last inequality we uses $tr_{\omega}\tilde\omega\ge e^2$. Using again $tr_{\omega}\tilde\omega\ge e^2$ we know $\log tr_{\omega}\tilde\omega\le(tr_{\omega}\tilde\omega)^{1/2}$, so we arrive at
\begin{align}
tr_\omega\tilde\omega\le\frac{C(tr_{\omega}\tilde\omega)^{1/2}}{\phi^2}\nonumber,
\end{align}
which gives, at $(\bar x,\bar t)$,
\begin{align}
tr_\omega\tilde\omega\le\frac{C}{\phi^4}\nonumber.
\end{align}
Therefore,
\begin{align}
P(\bar x,\bar t)&=\phi\cdot(\tau^{-1}\log tr_{\omega}\tilde\omega+u)\nonumber\\
&\le\phi\cdot\tau^{-1}\log tr_{\omega}\tilde\omega+C\nonumber\\
&\le\phi\cdot\log\frac{1}{\phi^4}+C.
\end{align}
Since the function $s\log\frac{1}{s^4}$ is uniformly bounded for $s\in(0,1]$, we get a uniform constant $C\ge1$ such that
\begin{align}
P(\bar x,\bar t)\le C\nonumber.
\end{align}
Since $C$ does not depend on the choice of $T$, there holds on $X\times[0,\infty)$ that
\begin{equation}
P\le C\nonumber.
\end{equation}
Recall $\phi\equiv1$ on $U'$, then combining the bound of $R$ in \eqref{ineq_5.1} we obtain a constant $C\ge1$ such that on $U'\times[0,\infty)$,
\begin{equation}\label{ineq_11}
tr_\omega\tilde\omega\le e^{C\cdot\tau}.
\end{equation}
From \eqref{vol_bd} by \cite{ST07, ST12} we know the volume forms $\omega^n$ and $\tilde\omega^n$ are uniformly equivalent on $X$, so by possibly increasing $C$ if necessary we have on $U'\times[0,\infty)$,
\begin{equation}\label{ineq_12}
tr_{\tilde\omega}\omega\le C\cdot e^{(n-1)C\cdot\tau}.
\end{equation}
Combining \eqref{ineq_11} and \eqref{ineq_12}, we find a constant $A\ge1$ such that on $U'\times[0,\infty)$,
\begin{equation}\label{ineq_13}
e^{-A\cdot\tau}\tilde\omega\le\omega\le e^{A\cdot\tau}\tilde\omega.
\end{equation}
The proof is completed.
\end{proof}

\begin{proof}[Proof of Theorem \ref{blowuprate}]
For case (i), i.e. $X_y$ is a regular fiber, by Fong-Z. Zhang \cite[Theorem 1.1 and Section 6]{FZ} (see also \cite{To} for the elliptic analogue) we know any two solutions are uniformly equivalent near $X_y$. Then we can choose $\sigma$ in Theorem \ref{tech.thm} to be a constant, and so (i) follows.
\par For case (ii), i.e. $X_y$ is a singular fiber, by Proposition \ref{equivalence} we can choose $\sigma$ in Theorem \ref{tech.thm} to be of the form $e^{A\tau}$, and so (ii) follows.
\end{proof}

\begin{proof}[Proof of Theorem \ref{ind.1}]
By choosing $\tau$ in Theorem \ref{blowuprate} to be a constant, the theorem follows easily: if $y \not\in W_{\textup{II}}(\tilde{\omega}(t))$, then there exist an open neighborhood $U$ of $X_y$ and a constant $C \geq 1$ such that on $U \times [0,\infty)$ we have:
\[\sup_{U}|Rm(\tilde\omega(t))|_{\tilde\omega(t)}\le C.\]
Proposition \ref{equivalence} shows there exist an open neighborhood $U' \subset\subset U$ of $X_y$ and a constant $A \geq 1$ such that on a $U' \times [0,\infty)$ we have:
\[e^{-AC}\tilde{\omega}(t) \leq \omega(t) \leq e^{AC}\tilde{\omega}(t).\]
Finally, we apply Theorem \ref{tech.thm} with $\sigma(t) \equiv e^{AC}$ to get that $y \not\in W_{\textup{II}}(\omega(t))$. It completes the proof.

\end{proof}

\section{Blow-up rate estimates}
\label{sect:e^t}
In this final section, we shall first prove the following general lemma, and then apply it to estimate curvature blow-up rate in certain settings.
\begin{lem}\label{lem_7.1}
Let $B_r$, $r>0$, be the ball in $\mathbb C^n$ defined by $B_r=\{(z^1,...,z^n)\in\mathbb C^n||z^1|^2+\ldots+|z^n|^2<r^2\}$, and $\omega_E$ the standard Euclidean metric on $\mathbb C^n$. Fix a smooth family $\tilde\omega=\tilde\omega(t), t\in[0,\infty)$, of K\"ahler metrics on $B_1$ satisfying
\begin{itemize}
\item[(1)] for any $t\ge0$, $\tilde\omega(t)$ is a flat metric;
\item[(2)] there is a constant $C$ such that $\partial_t\tilde\omega\le C\tilde\omega$ on $B_1\times[0,\infty)$;
\item[(3)] there exist an increasing differentiable function $\alpha(t) : [0, \infty) \to [1, \infty)$ and a constant $C\ge1$ such that $\tilde\omega(t)\ge C^{-1}\alpha^{-1}\omega_E$ on $B_1\times[0,\infty)$.
\end{itemize}
Let $\omega=\omega(t)$ be a solution to the K\"ahler-Ricci flow \eqref{nkrf} on $B_1\times[0,\infty)$. Assume there is an increasing differentiable function $\beta(t) : [0, \infty) \to [1, \infty)$ such that
\begin{equation}\label{bounds}
\tilde\omega\le\omega\le\beta\cdot\tilde\omega \quad \text{ on }  B_1\times[0,\infty),
\end{equation}
then we have a constant $C\ge1$ such that on $B_{1/2}\times[0,\infty)$,
$$|Rm(\omega)|_\omega\le C\alpha\cdot\beta.$$
\end{lem}
\begin{proof}
The proof is similar to our above discussions. The first step is to obtain a local Calabi's $C^3$-estimate, which is very similar to \cite[Proposition 2.7]{TWY}. Since we are in a setting more general than \cite[Proposition 2.7]{TWY}, we will provide some details. As before, define a tensor $\Psi=(\Psi_{ij}^k)$ by $\Psi_{ij}^k:=\Gamma_{ij}^k-\tilde\Gamma^k_{ij}$, where $\Gamma_{ij}^k$ (resp. $\tilde\Gamma_{ij}^k$) is the Christoffel symbols of $\omega(t)$ (resp. $\tilde\omega(t)$), and $S=|\Psi|_\omega^2$. Given conditions (3) we may fix a cutoff function $\phi$ on $B_1$, which is compactly supported on $B_1$, identically equals to $1$ on $B_{2/3}$, and satisfies, for a constant $C\ge1$, on $\overline B_1\times[0,\infty)$,
\begin{align}
|\partial\phi|_{\omega(t)}^2+|\Delta_{\omega(t)}\phi|\le C\alpha.
\end{align}
Then by direct computation we have
\begin{align}\label{S_e^t_0}
(\partial_t-\Delta_\omega)S&=S-|\nabla\Psi|_\omega^2-|\overline \nabla\Psi|_\omega^2\nonumber\\
&=S-|\nabla\Psi|_\omega^2-|Rm(\omega)|_\omega^2
\end{align}
and hence
\begin{align}\label{S_e^t}
(\partial_t-\Delta_\omega)(\phi^2S)\le C\alpha S.
\end{align}

On the other hand, given conditions (1), (2) and \eqref{bounds} we can use direct computation to get
\begin{align}
(\partial_t-\Delta_\omega)tr_\omega\tilde\omega&=tr_\omega(\tilde\omega+\partial_t\tilde\omega)+g^{\bar ji}g^{\bar qp}\tilde R_{i\bar jp\bar q}-g^{\bar ji}g^{\bar qp}\tilde g^{\bar ba}\nabla_i\tilde g_{p\bar b}\nabla_{\bar j}\tilde g_{a\bar q}\nonumber\\
&=tr_\omega(\tilde\omega+\partial_t\tilde\omega)+g^{\bar ji}g^{\bar qp}\tilde R_{i\bar jp\bar q}-
g^{\bar ji}g^{\bar qp}\tilde g_{d\bar e}\Psi^d_{ip}\overline{\Psi^e_{jq}}\nonumber\\
&\le C-\beta^{-1}S.
\end{align}
Note that \eqref{S_e^t} implies
\begin{align}\label{S_e^t.1}
(\partial_t-\Delta_\omega)(\phi^2\alpha^{-1}\beta^{-1}S)\le C\beta^{-1}S.
\end{align}
Therefore, choosing a sufficiently large constant $C$ gives
$$(\partial_t-\Delta_\omega)(\phi^2\alpha^{-1}\beta^{-1}S+Ctr_\omega\tilde\omega)\le-\beta^{-1}S+C.$$
Now by apply the maximum principle, we conclude that $\phi^2\alpha^{-1}\beta^{-1}S+Ctr_\omega\tilde\omega$ is a bounded function on $B_1\times[0,\infty)$ (here we have used $tr_\omega\tilde\omega$ is uniformly bounded by the first inequality in \eqref{bounds}), and hence 
\begin{equation}\label{S_e^t.1'}
S\le C\alpha\cdot\beta  \quad \text{ on } B_{2/3}\times[0,\infty).
\end{equation}
Now, we are going to bound $Rm(\omega)$. Combining the above \eqref{S_e^t.1'} and \eqref{S_e^t_0}, we have on $B_{2/3}\times[0,\infty)$,
\begin{align}\label{S_e^t_2}
(\partial_t-\Delta_\omega)(\alpha^{-1}\beta^{-1}S)\le C-\alpha^{-1}\beta^{-1}(|\nabla\Psi|_\omega^2+|\overline \nabla\Psi|_\omega^2)
\end{align}
and 
\begin{align}\label{S_e^t_3}
(\partial_t-\Delta_\omega)(\alpha^{-1}\beta^{-1}S)\le C-\alpha^{-1}\beta^{-1}|Rm(\omega)|_\omega^2.
\end{align}
Next, fix a cutoff function $\rho$ on $B_{2/3}$, which is compactly supported on $B_{2/3}$, identically equals to $1$ on $B_{1/2}$, and satisfies, for a constant $C\ge1$, on $\overline B_{2/3}\times[0,\infty)$,
\begin{align}
|\partial\rho|_{\omega(t)}^2+|\Delta_{\omega(t)}\rho|\le C\alpha.
\end{align}
Then we can use very similar analysis in Section \ref{Rm.Bd} to see that, for two sufficiently large constants $B, C$ there holds on $B_{2/3}\times[0,\infty)$,
\begin{align*}
& (\partial_t-\Delta_\omega)\left(\rho^2\frac{|Rm(\omega)|_\omega^2}{B-\alpha^{-1}\beta^{-1}S}\right)\\
& \le C\alpha|Rm(\omega)|_\omega^2+\frac{\rho^2}{(B-\alpha^{-1}\beta^{-1}S)^2}|Rm(\omega)|_\omega^3(BC-\frac12\alpha^{-1}\beta^{-1}|Rm(\omega)|_\omega)\nonumber
\end{align*}
and hence
\begin{align}
&(\partial_t-\Delta_\omega)\left(\rho^2\alpha^{-2}\beta^{-2}\frac{|Rm(\omega)|_\omega^2}{B-\alpha^{-1}\beta^{-1}S}\right)\nonumber\\
&\le C\alpha^{-1}\beta^{-2}|Rm(\omega)|_\omega^2+\frac{\rho^2\alpha^{-2}\beta^{-2}}{(B-\alpha^{-1}\beta^{-1}S)^2}|Rm(\omega)|_\omega^3(BC-\frac12\alpha^{-1}\beta^{-1}|Rm(\omega)|_\omega)\nonumber.
\end{align}
Together with \eqref{S_e^t_3}, we find a sufficiently large constant $C\ge1$ such that 
\begin{align}
&(\partial_t-\Delta_\omega)\left(\rho^2\alpha^{-2}\beta^{-2}\frac{|Rm(\omega)|_\omega^2}{B-\alpha^{-1}\beta^{-1}S}+C\alpha^{-1}\beta^{-1}S\right)\nonumber\\
&\le C-\alpha^{-1}\beta^{-1}|Rm(\omega)|_\omega^2+\frac{\rho^2\alpha^{-2}\beta^{-2}}{(B-\alpha^{-1}\beta^{-1}S)^2}|Rm(\omega)|_\omega^3(BC-\frac12\alpha^{-1}\beta^{-1}|Rm(\omega)|_\omega)\nonumber.
\end{align} 
Now applying the similar maximum principle arguments in Section \ref{Rm.Bd} completes the proof.
\end{proof}

We shall provide two applications of Lemma \ref{lem_7.1}. The first one is a proof of Theorem \ref{thm_Rm}.

\begin{proof}[Proof of Theorem \ref{thm_Rm}]
The first part $A^{-1}e^{t}\le\sup_{U}|Rm(\omega(t))|_{\omega(t)}$ is contained in \cite{ToZyg}. In fact, we may choose a regular fiber $X_y$ such that $U\cap X_y$ is an open subset in $X_y$. Then any Ricci-flat K\"ahler metric on $X_y$ can not be flat on $U\cap X_y$ (otherwise we get a flat metric on $X_y$, a contradiction). Therefore, the arguments in \cite[page 2941, (C, Case 1)]{ToZyg} can be applied to prove the lower estimate. 
\par Next we show the second part. To this end, by passing to a finite open cover, we may assume $U$ is a local chart centered at $x\in U$ and $U=\{(z^1,...,z^n)\in \mathbb C^n||z^1|^2+...+|z^n|^2<1\}$. Then similar to \cite[Proposition 2.7]{TWY}, we write $\mathbb C^{n}=\mathbb C^{k}\oplus\mathbb C^{n-k}$, and $\omega_{E}$, $\omega_{E}^{(k)}$, $\omega_{E}^{(n-k)}$ the Euclidean metrics on $\mathbb C^{n}$, $\mathbb C^{k}$, $\mathbb C^{n-k}$ respectively. Set $\omega_{E,t}:=\omega_{E}^{(k)}+e^{-t}\omega_{E}^{(n-k)}$. By Fong-Z. Zhang \cite[Theorem 1.1 and Section 6]{FZ} we have a constant $C\ge1$ such that on $U\times[0,\infty)$,
\begin{equation}
C^{-1}\omega_{E,t}\le\omega(t)\le C\omega_{E,t}.
\end{equation}
Then we can complete the proof by applying Lemma \ref{lem_7.1} with $\tilde\omega(t)=C^{-1}\omega_{E,t}$, $\alpha(t)=e^t$ and $\beta(t)\equiv C^2$.
\end{proof}

We observe one more application of Lemma \ref{lem_7.1}, which concerns the K\"ahler-Ricci flow on a compact K\"ahler manifold with semi-negative holomorphic sectional curvature. Pioneered by a conjecture of Yau in 1970s, compact K\"ahler manifolds with (semi-)negative holomorphic sectional curvature have been studied widely. More recently, after a breakthrough by Wu-Yau \cite{WY1}, there are many progresses on this subjects, including \cite{DT,HLW2,HLWZ,No,ToY}, just to mention a few. Here, by applying Lemma \ref{lem_7.1} we observe a curvature estimate of the K\"ahler-Ricci flow in terms of the existence of a K\"ahler metric with semi-negative holomorphic sectional curvature. 

\begin{prop}
Let $X$ be a compact K\"ahler manifold. Assume there exists a K\"ahler metric on $X$ with semi-negative holomorphic sectional curvature (and hence by \cite{ToY} $K_X$ is nef). Then for any long-time solution $\omega(t)$ to the K\"ahler-Ricci flow \eqref{nkrf} on X, there is a constant $C\ge1$ such that on $X\times[0,\infty)$,
$$|Rm(\omega(t))|_{\omega(t)}\le C\cdot e^{(n+1)t}.$$
\end{prop} 

We may point out that here we don't need to assume semi-ampleness of $K_X$.

\begin{proof}
Fix a K\"ahler metric $\omega_0$ on $X$ with semi-negative holomorphic sectional curvature. It was proved in \cite[Section 3]{No} that on $X\times[0,\infty)$,
$$\omega(t)\ge C^{-1}e^{-t}\omega_0.$$
Then using the uniform upper bound of volume form (see e.g. \cite[Corollary 2.3(ii)]{SW}) one gets
$$\omega(t)\le Ce^{(n-1)t}\omega_0.$$
Now, after passing to a local chart, we can apply Lemma \ref{lem_7.1} with $\tilde\omega(t)=C^{-1}e^{-t}\omega_E$, $\alpha(t)=e^t$ and $\beta(t)=C^2e^{nt}$ to conclude this proposition.
\end{proof}

\section*{Acknowledgements}
The authors thank professors Huai-Dong Cao, Jian Song, Gang Tian, Valentino Tosatti, Ben Weinkove and Zhenlei Zhang for interest and comments on a previous version of this paper. The second-named author also thanks professors Huai-Dong Cao and Gang Tian for constant encouragement and support and Wangjian Jian for valuable comments and discussions. Finally, the authors would like to thank the referees for their careful readings and very useful comments and suggestions.

\end{document}